\documentclass[12pt, letterpaper]{article}

\usepackage[utf8]{inputenc}

\usepackage[backend=biber,
style=numeric-comp, sorting=nyt, doi=false, isbn=false, date=year]{biblatex}

\AtEveryBibitem{\clearfield{url} \clearfield{note}}

\DeclareNameAlias{author}{family-given}
\DeclareNameAlias{editor}{family-given}

\AtBeginBibliography{%
}

\usepackage[pdftex,colorlinks=true,linkcolor=blue,citecolor=blue,urlcolor=red,unicode=true,hyperfootnotes=true,bookmarksnumbered]{hyperref}
\usepackage[margin=1truein]{geometry}

\usepackage{amsfonts}

\usepackage{wrapfig}

\usepackage{caption}
\usepackage{subcaption}

\usepackage{amsfonts, amssymb, amsmath, amsthm}
\usepackage{mathtools}
\usepackage{enumerate}
\usepackage{verbatim}

\usepackage{mathrsfs,amsfonts,mathtools}
\usepackage{amsmath}
\usepackage{amssymb}

\usepackage{amsthm}
\usepackage{thmtools}

\theoremstyle{plain}
\newtheorem{theorem}{Theorem}[section]
\newtheorem{lemma}[theorem]{Lemma}

\newtheorem{proposition}[theorem]{Proposition}
\newtheorem{corollary}[theorem]{Corollary}

\theoremstyle{definition} 
\newtheorem{example}[theorem]{Example}
\newtheorem{question}[theorem]{Question}

\newtheorem*{remark*}{Remark}
\newtheorem{definition}{Definition}
\newtheorem*{definition*}{Definition}

\renewcommand{\emptyset}{\varnothing}

\renewcommand{\epsilon}{\varepsilon}
\renewcommand{\phi}{\varphi}
\renewcommand{\kappa}{\varkappa}
\newcommand{\R}{\mathbb{R}}
\newcommand{\Q}{\mathbb{Q}}
\newcommand{\Z}{\mathbb{Z}}
\newcommand{\Nz}{\Z_{\ge0}}

\newcommand{\br}[1]{\left({#1}\right)}
\newcommand{\fbr}[1]{\left\{#1\right\}}
\newcommand{\lbr}[1]{\left|{#1}\right|}
\renewcommand{\sp}[1]{\langle{#1}\rangle}

\newcommand{\wt}[1]{\widetilde{#1}}

\newcommand{\mc}[1]{\mathcal{#1}}
\newcommand{\mb}[1]{\mathbf{#1}}
\newcommand{\conv}[2]{{#1}^{#2}}
\newcommand{\fconv}[2]{{#1}^{#2}}

\newcommand{\wsat}{\operatorname{\mathrm{wsat}}}
\newcommand{\wdir}{\operatorname{\mathrm{c-wsat}}}
\newcommand{\Wsat}{\mathrm{wSAT}}
\newcommand{\seq}{\subseteq}
\newcommand{\disjointcup}{\cup}
\newcommand{\bigdisjointcup}{\bigcup}
\newcommand{\timesprod}{\prod}

\newcommand{\tprod}[2]{K^{#1}_{#2}}

\DeclareMathOperator{\rk}{rk}
\DeclareMathOperator{\sgn}{sgn}
\DeclareMathOperator{\inv}{inv}
\DeclareMathOperator{\supp}{supp}
\DeclareMathOperator{\downcl}{\downarrow}
\DeclareMathOperator{\Conv}{Conv}
\DeclareMathOperator{\spann}{span}

\title{Weak saturation of tensor product of cliques}

\author{Nikolai Terekhov\thanks{Department of Discrete Mathematics, Moscow Institute of Physics and Technology, Dolgoprudny, Russia; nikolayterek@gmail.com
}}

\date{}

\begin{document}

\maketitle


\begin{abstract}
    Given two hypergraphs $G$ and $H$, the weak saturation number $\wsat(G,H)$ is the minimum number of edges in a spanning subhypergraph $F$ of $G$ such that the missing edges of $F$ can be added one at a time so that each added edge creates a copy of $H$.

    In this work, we determine weak saturation numbers for the case when $G$ and $H$ are tensor product of cliques, generalizing a result of Moshkovitz and Shapira (Journal of Combinatorial Theory, Series B, 2015), who found the exact values of $\wsat\br{K^d_{n_1,\ldots,n_d},\ K^d_{r_1,\ldots,r_d}}$.

    The proof also yields results for colored weak saturation numbers $\wdir(G,H)$ of colored hypergraphs $G$ and $H$, where the colorings of the copies of $H$ must be compatible with the coloring of $G$. We determine these numbers when $G$ and $H$ are unions of tensor product of cliques, generalizing a result of Bulavka, Tancer, and Tyomkyn (Combinatorica, 2023), who determined $\wdir(K^q_{n_1,\ldots,n_d}, K^q_{r_1,\ldots,r_d})$.

    Moreover, our proof allows us to generalize a result of  Balogh, Bollob\'{a}s, Morris, and Riordan (Journal of Combinatorial Theory, Series A, 2012) by determining  colored weak saturation numbers $\wdir(K^d_{n_1,\ldots,n_d},\fbr{K^d_{r_1,\ldots,r_d}}_{\mb{r}\in \mc{R}})$ for an arbitrary family $\mc{R}$. The quantity $\wdir(G,\mc{H})$ extends colored weak saturation by allowing, at each step, the creation of a colored copy of any hypergraph in the fixed family of hypergraphs $\mc{H}$.
\end{abstract}

\section{Introduction}

Cellular automata, introduced by von Neumann~\cite{1966Neumann} after Ulam's suggestion~\cite{1950Ulam}, are used to model a variety of processes in physics, chemistry, biology, and cryptography. Bollobás~\cite{1968Bollobas} introduced graph bootstrap percolation, which is a particular case of monotone cellular automata. Graph bootstrap percolation is a substantial generalisation of $r$-neighbourhood bootstrap percolation, which has applications in physics --- see, for example,~\cite{2003Adler,2002Fontes,2011Morris}.
Given hypergraphs $G$ and $H$, an $H$-bootstrap percolation process is a sequence of hypergraphs $F_0 \seq F_1 \seq \ldots \seq  F_m$ such that, for each $i\ge 1$, $F_i$ is obtained from $F_{i-1}$ by adding an edge that creates a new copy of $H$. A spanning subhypergraph $F$ of $G$ is called weakly \mbox{$H$-saturated} in $G$ if there exists an $H$-bootstrap percolation process $F=F_0 \seq F_1 \seq \ldots \seq F_m=G$. In this case, we write $F \in \Wsat(G, H)$. The minimum number of edges in such a hypergraph $F$ is denoted $\wsat(G, H)$ and is called the weak saturation number of $H$ in $G$.

Weak saturation numbers have been extensively studied
\cite{
    1982Frankl,
    1985Alon,
    1985Kalai,
    1991Erdos,
    1992Tuza,
    2001PikhurkoCount,
    2012BaloghLinearAlgebra,
    2018Morrison,
    2020Hambardzumyan,
    2023Shapira,
    2025TerekhovLinearAlgebra}.
However, in the general case, determining these numbers is quite difficult, and only partial results are known. The most effective method for finding weak saturation numbers is the linear algebraic method introduced in full generality by Kalai~\cite{1985Kalai}. This method has been used to determine the values of $\wsat(K^s_n, K^s_r)$~\cite{1984KalaiRigid, 1985Kalai,1982Frankl,1985Alon}, where $K^s_n$ denotes the complete $s$-uniform hypergraph on $n$ vertices; $\wsat(K^2_n, K^2_{t,t})$ and $\wsat(K^2_n, K^2_{t,t+1})$~\cite{1985Kalai,2021Kronenberg}, where $K^d_{r_1,\ldots,r_d}$ denotes the $d$-uniform complete $d$-partite hypergraph with the $i$-th part of size $r_i$; $\wsat(K^d_{n_1,\ldots,n_d}, K^d_{r_1,\ldots,r_d})$~\cite{1985Alon,2015Moshkovitz}; as well as weak saturation numbers for pyramids~\cite{2001PikhurkoExterior}.

The results presented above show that the linear algebraic method works particularly well for cliques and multipartite hypergraphs. It is therefore natural to consider the following family of hypergraphs, which generalizes both cliques and $d$-uniform complete $d$-partite hypergraphs.

\begin{definition}[\cite{2001PikhurkoExterior}]
    Let $G_1$ and $G_2$ be two hypergraphs with disjoint vertex sets. Their \emph{tensor product} $G_1 \otimes G_2$ is the hypergraph with vertex set
    $V(G_1 \otimes G_2) = V(G_1) \cup V(G_2)$
    and edge set
    $E(G_1 \otimes G_2) = \fbr{e_1 \cup e_2 \mid e_1 \in E(G_1), e_2 \in E(G_2)}$.
\end{definition}

\begin{definition}
    Let $d \ge 1$ be an integer, and let $\mb{s} \in \Nz^d$ and $\mb{r} \in \Nz^d$ be integer vectors. Define the hypergraph
    \[
    \tprod{\mb{s}}{\mb{r}} := \bigotimes_{i \in [d]} K^{s_i}_{[r_i] \times \fbr{i}},
    \]
    where $K^m_T$ is an $m$-uniform clique on the vertex set $T$, and $[n]$ denotes the set $\fbr{1,2,\ldots,n}$.
\end{definition}

The hypergraph $\tprod{\mb{s}}{\mb{r}}$ is closely related to the complete $\mb{s}$-layered hypergraph introduced by Pikhurko~\cite{2001PikhurkoExterior}, although our definition does not include the additional partition structure on the layers.

It is easy to see that this definition generalizes both cliques, which arise when $d=1$, and $d$-uniform complete $d$-partite hypergraphs, since $K^d_{r_1,\ldots,r_d} = \tprod{\mb{s}}{\br{r_1,\ldots,r_d}}$ with $\mb{s} = \br{1,1,\ldots,1}$.
One of the first results on weak saturation for such hypergraphs is due to Alon~\cite{1985Alon}, who determined $\wsat\br{K^d_{n,\ldots,n},\ K^d_{r,\ldots,r}}$. This result was later extended in~\cite{2015Moshkovitz} to $\wsat\br{K^d_{n,\ldots,n},\ K^d_{r_1,\ldots,r_d}}$ for all choices of parameters $\fbr{r_i}_{i\in[d]}$ and $n\ge 1$. In~\cite{2015Moshkovitz} the authors used Alon's result together with combinatorial arguments in their proof. In the present work, we obtain a purely algebraic proof of the following general result that subsumes the result from~\cite{2015Moshkovitz}.

\begin{theorem}\label{thrm:sym-Kn1n2-Kr1r2}
    Let $d \ge 1$ and $s \ge 1$ be integers, and let $\mb{n} \in \Nz^d$ and $\mb{r} \in \Nz^d$ be integer vectors. Suppose that $n_i \ge s$ and $r_i \ge s$ for all $i \in [d]$. Define $\mb{s} \in \Nz^d$ by $s_i = s$ for all $i \in [d]$. Then
    \[
    \wsat(\tprod{\mb{s}}{\mb{n}}, \tprod{\mb{s}}{\mb{r}}) = \prod_{i\in[d]} \binom{n_i}{s} - \tilde{q}(\mb{n}, s, \mb{r}),
    \]
    where\footnote{For a finite set $A$, $\binom{A}{s}$ denotes $\fbr{B \subseteq A\ \Big|\ \lbr{B} = s}$, and $S_d$ denotes the set of all permutations of $[d]$.}
    \[
        \tilde{q}(\mb{n}, s, \mb{r}) = \lbr{\fbr{T \in \timesprod_{i \in [d]} \binom{[n_i]}{s}\ \middle|\ \exists \sigma \in S_d\ \forall i \in [d]\ T_i \in \binom{[n_i] \setminus [r_{\sigma(i)} - s]}{s}}}.
    \]
\end{theorem}
The value of $\wsat(\tprod{\mb{s}}{\mb{n}}, \tprod{\mb{s}}{\mb{r}})$ in the theorem arises naturally from the upper bound construction presented in Subsection~\ref{subsec:upper-bound-largest-s} for a more general setting.

Our proof uses a reduction to colored weak saturation numbers $\wdir(G,H)$, in which the host hypergraph $G$ and the pattern hypergraph $H$ are additionally equipped with colorings (not necessarily proper) in $d \ge 1$ colors. This notion generalizes the usual weak saturation numbers by requiring that, in the $H$-bootstrap percolation process, the coloring of each copy of $H$ in $G$ must be compatible with the coloring of $G$. A formal definition is given in Section~\ref{sec:main-thrm}. For a hypergraph $\tprod{\mb{s}}{\mb{r}}$, we use the coloring $c:\bigcup_{i\in[d]}[r_i]\times\fbr{i}\to [d]$ such that $c([r_i]\times\fbr{i})=\fbr{i}$ for all $i\in[d]$.

Colored weak saturation numbers are, in some sense, similar to the weak saturation numbers for layered hypergraphs introduced in~\cite{2001PikhurkoExterior}. The difference is that, in our setting, we do not impose any conditions relating the edges of a hypergraph to its coloring.

In order to overview known results on colored weak saturation, we introduce the following family of hypergraphs.

\begin{definition}
Let $G_1$ and $G_2$ be two hypergraphs. Their \emph{union} is the hypergraph $G_1 \cup G_2$ with vertex set $V(G_1 \cup G_2) = V(G_1) \cup V(G_2)$ and edge set $E(G_1 \cup G_2) = E(G_1) \cup E(G_2)$.
\end{definition}

\begin{definition}
    Let $d \ge 1$ be an integer, let $\mb{r} \in \Nz^d$ be an integer vector, and let $\mc{S} \subseteq \Nz^d$ be a non-empty finite family of integer vectors. Define the hypergraph
    \[
    \tprod{\mc{S}}{\mb{r}} := \bigcup_{\mb{s} \in \mc{S}} \tprod{\mb{s}}{\mb{r}}.
    \]
\end{definition}

This definition generalizes the $q$-uniform complete $d$-partite hypergraphs, since $K^q_{r_1,\ldots,r_d} = \tprod{\mc{S}}{\br{r_1,\ldots,r_d}}$ with $\mc{S} = \fbr{\mb{s} \in \fbr{0,1}^d \mid \sum_{i \in [d]} s_i = q}$, where $K^q_{r_1,\ldots,r_d}$ denotes the $q$-uniform complete $d$-partite hypergraph with the $i$-th part of size $r_i$.

The value of $\wdir(\tprod{\mb{s}}{\mb{n}}, \tprod{\mb{s}}{\mb{r}})$ can be deduced either from Alon's version of the two families theorem~\cite{1985Alon} or from Pikhurko's result for layered hypergraphs~\cite{2001PikhurkoExterior} for all choices of parameters $\mb{s}, \mb{n}, \mb{r} \in \Nz^d$. Bulavka, Tancer, and Tyomkyn~\cite{2023Bulavka} determined
$\wdir(\tprod{\mc{S}}{\mb{n}}, \tprod{\mc{S}}{\mb{r}})$
for
$\mc{S} = \fbr{\mb{s} \in \fbr{0,1}^d \mid \sum_{i \in [d]} s_i = q}$,
for all choices of parameters $2\le q \le d$ and $\mb{n}, \mb{r} \in \Nz^d$ such that
$n_i \ge r_i \ge 1$ for all $i \in [d]$.
In this paper, we substantially generalize their result by determining
$\wdir(\tprod{\mc{S}}{\mb{n}}, \tprod{\mc{S}}{\mb{r}})$
for all finite non-empty families $\mc{S}$.

\begin{theorem}\label{thrm:wdir-single-r}
    Let $d \ge 1$ be an integer, let $\mb{r} \in \Nz^d$ and $\mb{n} \in \Nz^d$ be integer vectors, and let $\mc{S} \subseteq \Nz^d$ be a non-empty finite family of integer vectors. Suppose that for all $\mb{s} \in \mc{S}$ and all $i \in [d]$ we have $n_i \ge r_i \ge s_i$. Define
\[
\downcl \mc{S} := \fbr{\mb{m} \in \Nz^d \mid \exists \mb{s} \in \mc{S}\ \forall i \in [d]\ m_i \le s_i}.
\]
Then
\[
\wdir(\tprod{\mc{S}}{\mb{n}}, \tprod{\mc{S}}{\mb{r}})
= \sum_{\mb{s} \in \mc{S}} \prod_{i \in [d]} \binom{n_i}{s_i}
- \sum_{\mb{m} \in \downcl \mc{S}} \prod_{\substack{i = 1 \\ m_i \ne 0}}^{d}
\binom{m_i - 1 + n_i - r_i}{m_i}.
\]
\end{theorem}

Colored weak saturation can be used to deduce results about uncolored weak saturation. This is possible because, in some cases, an uncolored copy of a hypergraph $H$ arising in the $H$-bootstrap percolation process can be viewed as a colored copy of a hypergraph belonging to a suitable family $\mc{H}$. In such cases, the value of $\wsat(\tprod{\mb{s}}{\mb{n}},H)$ is equal to the value of $\wdir(\tprod{\mb{s}}{\mb{n}},\mc{H})$, where $\wdir(G,\mc{H})$ generalizes colored weak saturation by allowing, at each step, the creation of a colored copy of any hypergraph in $\mc{H}$. A formal definition is given in Section~\ref{sec:main-thrm}.


In order to state our result on colored weak saturation numbers for a family of pattern hypergraphs, we introduce the following families.
\begin{definition}
    Let $d \ge 1$ be an integer, let $\mb{s}\in \Nz^d$ be integer vector, and let $\mc{R} \subseteq \Nz^d$ be a non-empty finite family of integer vectors. Define the family of hypergraphs
    \[
        \tprod{\mb{s}}{\mc{R}}:=\fbr{\tprod{\mb{s}}{\mb{r}}\mid \mb{r}\in\mc{R}}.
    \]
\end{definition}

For such families of hypergraphs, the result of \cite{2012BaloghLinearAlgebra}, after reformulation in terms of colored weak saturation, gives the value of $\wdir(\tprod{\mb{s}}{\mb{n}}, \tprod{\mb{s}}{\mc{R}(k, \mb{t})})$ for $d\ge 1$, $\mb{n}\in \Nz^d$, $k\in\Nz$, $\mb{t}\in \Nz^d$, $\mb{s}=(1,1,\ldots,1)$, and
\[\mc{R}(k, \mb{t}) = \fbr{\mb{r}\in\Nz^d\ \Big|\ \bigl|\fbr{i\in [d]\mid r_i=t_i}\bigr|=k\text{ and }\bigl|\fbr{i\in [d]\mid r_i=1}\bigr|=d-k},\]
where $1\le k \le d$ and $n_i\ge t_i\ge 2$ for all $i\in[d]$. In this work, we substantially generalize their result by determining $\wdir(\tprod{\mb{s}}{\mb{n}}, \tprod{\mb{s}}{\mc{R}})$ for all finite non-empty families $\mc{R}$.

\begin{theorem}\label{thrm:wdir-multiply-r}
    Let $d \ge 1$ be an integer, let $\mb{n} \in \Nz^d$ and $\mb{s}\in \Nz^d$ be integer vectors, and let $\mc{R} \subseteq \Nz^d$ be a non-empty finite family of integer vectors. Suppose that for all $\mb{r} \in \mc{R}$ and $i \in [d]$ we have $n_i \ge s_i$ and $r_i \ge s_i$. Define the family of hypergraphs
\[
    \tprod{\mb{s}}{\mc{R}}:=\fbr{\tprod{\mb{s}}{\mb{r}}\mid \mb{r}\in\mc{R}}.
\]
Then
\[
    \wdir(\tprod{\mb{s}}{\mb{n}}, \tprod{\mb{s}}{\mc{R}}) = \prod_{i\in[d]} \binom{n_i}{s_i} - q(\mb{n}, \mb{s}, \mc{R}),
    \]
where
\[
    q(\mb{n}, \mb{s}, \mc{R}) = \lbr{\fbr{T \in \timesprod_{i \in [d]} \binom{[n_i]}{s_i}\ \middle|\ \exists r\in\mc{R}\ \forall i \in [d]\ T_i \in \binom{[n_i] \setminus [r_{i} - s_i]}{s_i}}}.
\]
\end{theorem}

\paragraph{Proof Strategy}

Theorem~\ref{thrm:sym-Kn1n2-Kr1r2} follows from Theorem~\ref{thrm:wdir-multiply-r} because, for that choice of parameters, there exists a suitable family $\mc{R}$ such that an uncolored copy of $\tprod{\mb{s}}{\mb{r}}$ in $\tprod{\mb{s}}{\mb{n}}$ can be viewed as a colored copy of a hypergraph from $\tprod{\mb{s}}{\mc{R}}$.

Theorems~\ref{thrm:wdir-single-r} and~\ref{thrm:wdir-multiply-r} follow from a more general result on colored weak saturation numbers, namely Theorem~\ref{thrm:main-theorem}. The essential ingredient of the proof of the lower bound in Theorem~\ref{thrm:main-theorem} is the linear algebraic technique that was used in~\cite{2023Bulavka}. We show that this technique extends to graph bootstrap percolation processes that exploit pattern hypergraphs from a given family $\mc{H}$, rather than a single hypergraph $H$; this extension is essential for Theorem~\ref{thrm:wdir-multiply-r}. We also show that the technique can be adapted to arbitrary families $\mc{S}$, which is needed for Theorem~\ref{thrm:wdir-single-r}.

\paragraph{Paper Structure}

In Section~\ref{sec:main-thrm} we state the main theorem on $\wdir(\tprod{\mc{S}}{\mb{n}}, \tprod{\mc{S}}{\mb{r}})$ in full generality; Theorems~\ref{thrm:wdir-single-r} and~\ref{thrm:wdir-multiply-r} appear as special cases. In Section~\ref{sec:preliminaries} we provide the background on the exterior algebra needed for the linear algebraic technique from~\cite{2023Bulavka}. In Section~\ref{sec:lower-bound-main-thrm} we prove the lower bound in the main theorem, and in Section~\ref{sec:upper-bound-main-thrm} we establish the corresponding upper bound. In Section~\ref{sec:wsat-eq-wdir} we derive corollaries for uncolored weak saturation numbers, including Theorem~\ref{thrm:sym-Kn1n2-Kr1r2}, which is a special case of Corollary~\ref{corollary:wsat-constant-single-s}. In Section~\ref{sec:conclusions} we discuss limitations of the main theorem and related open problems.

\paragraph{Notation}  For an integer $n \ge 1$, let $[n]$ denote the set $\fbr{1,2,\ldots,n}$. For an integer $n \le 0$, we define $[n] = \emptyset$. For a set $T$ and an integer $k \ge 0$, let
$\binom{T}{k} = \fbr{S \subseteq T\ |\ \lbr{S} = k}$.
For $k < 0$, we define $\binom{T}{k} = \emptyset$. For a set $T$ and an integer $k \ge 0$, define the hypergraph $K_T^k$ with vertex set $V(K_T^k) = T$ and edge set $E(K_T^k) = \binom{T}{k}$.

\section{Main Theorem}\label{sec:main-thrm}

We begin by giving a formal definition of weak saturation for a family of pattern hypergraphs.

\begin{definition}
    Let $d \ge 1$ be an integer, let $G$ be a hypergraph, and let $\mc{H}$ be a non-empty family of hypergraphs. A spanning subhypergraph $F$ of $G$ is called \emph{weakly $\mc{H}$-saturated} in $G$ if there exists an ordering $e_1, e_2, \ldots, e_k$ of the edges in $E(G) \setminus E(F)$ such that each edge $e_i$ belongs to a copy $\wt{H}_i$ of some $G \in \mc{H}$ in the hypergraph
    $F \cup \fbr{e_1, \ldots, e_i}$.
    The minimum number of edges in such a subhypergraph $F$ is denoted by $\wsat(G, \mc{H})$ and is called the \emph{uncolored weak saturation number}, or simply the \emph{weak saturation number}, of $\mc{H}$ in $G$.
\end{definition}

To define colored weak saturation, we first introduce the notion of a colored hypergraph.

\begin{definition}
Let $d\ge 1$ be an integer. A $d$-colored hypergraph is a pair $(G, c)$, where $G$ is a hypergraph (not necessarily uniform) and $c$ is a coloring (not necessarily proper), that is, an arbitrary map from $V(G)$ to $[d]$.
\end{definition}


\begin{definition}
Let $d\ge 1$ be an integer, and let $(H, t)$ and $(G, c)$ be $d$-colored hypergraphs. A colored copy of $(H, t)$ in $(G, c)$ is a subhypergraph $\wt{H} \subseteq G$ with the coloring $c|_{V(\wt{H})}$ such that there exists a bijection $f: V(\wt{H}) \to V(H)$ satisfying
\(
    \forall S \subseteq V(\wt{H}),\  S \in E(\wt{H}) \iff f(S) \in E(H),
\)
and
\(
    \forall v \in V(\wt{H}),\  c(v) = t(f(v)).
\)
\end{definition}

\begin{definition}
    Let $d\ge 1$ be an integer, and let $(F, t)$ and $(G, c)$ be $d$-colored hypergraphs. We say that $(F, t)$ is a \emph{colored spanning subhypergraph} of $(G, c)$ if $V(F) = V(G)$, $t = c$, and $E(F) \seq E(G)$.
\end{definition}

The definition of colored weak saturation mirrors that of uncolored weak saturation, but includes additional color constraints.
\begin{definition}
    Let $d \ge 1$ be an integer, let $(G, c)$ be a $d$-colored hypergraph, and let $\mc{H}$ be a non-empty family of $d$-colored hypergraphs. A colored spanning subhypergraph $(F,c)$ of $(G,c)$ is called \emph{weakly $\mc{H}$-saturated} in $(G, c)$ if there exists an ordering $e_1, e_2, \ldots, e_k$ of the edges in $E(G) \setminus E(F)$ such that each edge $e_i$ belongs to a colored copy $(\wt{H}_i,\tilde{t})$ of some $(H, t) \in \mc{H}$ in the $d$-colored hypergraph
    $(F \cup \fbr{e_1, \ldots, e_i}, c)$.
    The minimum number of edges in such a colored subhypergraph $(F,c)$ is denoted by $\wdir((G, c), \mc{H})$ and is called the \emph{colored weak saturation number} of $\mc{H}$ in $(G, c)$.
\end{definition}

We consider colored weak saturation numbers for the following colored hypergraphs.

\begin{definition}
    Let $d\ge 1$ be an integer, let $\mb{r} \in \Nz^d$ be an integer vector, and let $\mc{S} \subseteq \Nz^d$ be a non-empty finite family of integer vectors. Define a coloring $c_{\mb{r}}: \bigcup_{i \in [d]} [r_i] \times \fbr{i} \to [d]$ on $\tprod{\mc{S}}{\mb{r}}$
by setting $c_{\mb{r}}([r_i] \times \fbr{i}) = \fbr{i}$.

To simplify notation, whenever $\tprod{\mc{S}}{\mb{r}}$ is viewed as a colored hypergraph, it is understood to carry this coloring.
\end{definition}


\begin{definition}
Let $d \ge 1$ be an integer, and let $\mc{S} \subseteq \Nz^d$ and $\mc{R} \subseteq \Nz^d$ be a non-empty finite families of integer vectors. Define $\tprod{\mc{S}}{\mc{R}}$ to be the family of hypergraphs
$\fbr{\tprod{\mc{S}}{\mb{r}} \mid \mb{r} \in \mc{R}}$.
\end{definition}

Our main result on colored weak saturation is the following theorem.

\begin{theorem} \label{thrm:main-theorem}
    Let $d \ge 1$ be an integer, let $\mb{n} \in \Nz^d$ be an integer vector, and let $\mc{S} \subseteq \Nz^d$ and $\mc{R} \subseteq \Nz^d$ be non-empty finite families of integer vectors. Suppose that for all $\mb{s} \in \mc{S}$, $\mb{r} \in \mc{R}$, and $i \in [d]$, we have $n_i \ge s_i$ and $r_i \ge s_i$. Define
\begin{align*}
    \downcl \mc{S} &:= \fbr{\mb{m} \in \Nz^d \mid \exists \mb{s} \in \mc{S} \ \forall i \in [d]\ m_i \le s_i}, \\
    \mc{R}(\mb{n}) &:= \fbr{\mb{r} \in \mc{R} \mid \forall i \in [d]\ r_i \le n_i}.
\end{align*}
Then
\begin{equation}\label{eq:main-theorem}
    \wdir(\tprod{\mc{S}}{\mb{n}}, \tprod{\mc{S}}{\mc{R}}) \ge \sum_{\mb{s} \in \mc{S}} \prod_{i \in [d]} \binom{n_i}{s_i} - q(\mb{n}, \mc{S}, \mc{R}),
\end{equation}
where
\[
q(\mb{n}, \mc{S}, \mc{R}) = \sum_{\mb{m} \in \downcl \mc{S}} \sum_{\emptyset \ne \mc{Q} \subseteq \mc{R}(\mb{n})} (-1)^{|\mc{Q}| + 1} \prod_{\substack{i = 1 \\ m_i \ne 0}}^{d} \binom{m_i - 1 + n_i - \max_{\mb{r} \in \mc{Q}} r_i}{m_i}.
\]
Moreover, the bound \eqref{eq:main-theorem} is tight if at least one of the following holds:
\begin{itemize}
    \item There exists $\tilde{\mb{r}} \in \mc{R}$ such that for all $\mb{r} \in \mc{R}$ and $i \in [d]$ we have $r_i \ge \tilde{r}_i$.
    \item There exists $\tilde{\mb{s}} \in \mc{S}$ such that for all $\mb{s} \in \mc{S}$ and $i \in [d]$ we have $s_i \le \tilde{s}_i$.
\end{itemize}
\end{theorem}

\section{Preliminaries}\label{sec:preliminaries}

Sections \ref{subsec:kalai-method} and \ref{subsec:lower-bound-via-ext} describe the linear algebraic method that we use in Section~\ref{sec:lower-bound-main-thrm} to obtain a lower bound on $ \wdir(\tprod{\mc{S}}{\mb{n}}, \tprod{\mc{S}}{\mc{R}})$.

Sections \ref{subsec:exterior-algebra}, \ref{subsec:generic-basis}, \ref{subsec:colorful-generic-basis}, and \ref{subsec:left-interior} provide background on the exterior algebra and the operations therein, which are necessary for the application of the linear algebraic technique from \cite{2023Bulavka}. The proofs largely follow \cite{2023Bulavka}, and we also make use of \cite{2019Rosen}. We additionally give an alternative, more direct proof of the existence of generic matrices (Lemma~\ref{lemma:generic-matrix-exists}), first established in \cite{1984KalaiPattern}.

The presentation aims to be self-contained; thus, definitions are given in a form suitable for proofs, though equivalent to standard ones. We only consider finite-dimensional vector spaces.

\subsection{Linear Algebraic Method} \label{subsec:kalai-method}

The following linear algebraic method was introduced by Kalai \cite{1985Kalai} for uncolored weak saturation numbers in the more general context of matroids. Here we restrict ourselves to the vector-space version, which is sufficient for our purposes.

\begin{lemma}\label{lemma:lower-bound-kalai}
Let $d \ge 1$ be an integer, let $(G, c)$ be a $d$-colored hypergraph, let $V$ be a vector space, and let $\mc{H}$ be a non-empty family of $d$-colored hypergraphs. Let $f: E(G) \to V$ be a map such that for every $(H, t) \in \mc{H}$ and every colored copy $(\wt{H},\tilde{t})$ of $(H, t)$ in $G$, the image of every edge $e \in E(\wt{H})$ lies in the linear span of the images of the other edges in $E(\wt{H})$, i.e.,
\begin{equation}\label{eq:kalai-la}
    \forall e \in E(\wt{H}) \quad \rk\left(f\left(E(\wt{H}) \setminus \fbr{e}\right)\right) = \rk\left(f\left(E(\wt{H})\right)\right).
\end{equation}
Then
\[
\wdir((G, c), \mc{H}) \ge \rk\left(f(E(G))\right).
\]
\end{lemma}

\begin{proof}
Let colored spanning subhypergraph $(F,c)$ of $G$ be a weakly $\mc{H}$-saturated in $(G,c)$ with $|E(F)| = \wdir((G, c), \mc{H})$. Assume that the edges are added in order $e_1, \ldots, e_s$, and each $e_i$ belongs to a colored copy $(H_i,t_i)$ of some $(H, t) \in \mc{H}$ in $(F_i, c|_{F_i})$, where $F_i = F \cup \fbr{e_1, \ldots, e_i}$ and $F_0 = F$. Then for each $i$, due to condition \eqref{eq:kalai-la}, the edge $e_i$ is linearly dependent on the images of the other edges in $E(H_i) \setminus \fbr{e_i} \subseteq E(F_{i-1})$, so
\[
\rk(f(E(F_{i-1}))) = \rk(f(E(F_i))).
\]
Hence,
\[
\rk(f(E(F_0))) = \rk(f(E(F_s))) = \rk(f(E(G))),
\]
and therefore
\[
\wdir((G, c), \mc{H}) = |E(F_0)| \ge |f(E(F_0))| \ge \rk(f(E(F_0))) = \rk(f(E(G))).
\]
\end{proof}

\subsection{Exterior algebra} \label{subsec:exterior-algebra}

\begin{definition}\label{def:exterior-algebra}
Let $N$ be a finite set of size $n$ equipped with a linear order. Let $V$ be a real vector space of dimension $n$ with a fixed orthonormal basis $\fbr{e_i}_{i \in N}$. Let $W$ be a real vector space of dimension $2^n$ with a fixed orthonormal basis $\fbr{e'_i}_{i \in [2^n]}$.

Let $f: 2^N \to [2^n]$ be a bijection that enumerates $2^N$ in lexicographic order. For each $S \subseteq N$, define $e_S := e'_{f(S)}$. For each $i \in N$, we identify $e_i$ with $e_{\fbr{i}}$, and identify $V$ with the subspace $\mathrm{span}(\fbr{e_{\fbr{i}}}_{i \in N}) \subseteq W$.

Define a bilinear operation $\wedge : W \times W \to W$ acting on basis vectors as follows:
\[
\forall S, T \subseteq N \quad e_S \wedge e_T =
\begin{cases}
\sgn(S, T)\, e_{S \cup T}, &\text{if } S \cap T = \emptyset;\\
0, & \text{otherwise},
\end{cases}
\]
where $\sgn(S, T) := (-1)^{\inv(S, T)}$ and $\inv(S, T) = |\fbr{(s, t) \in S \times T \mid s > t}|$.

We refer to $W$ as the \textit{exterior algebra} over $V$ and denote it by $\bigwedge V$. The operation $\wedge$ is called the \emph{exterior product}.
\end{definition}

\begin{remark*}
The notation $\bigwedge V$ does not reflect its dependence on the choice of basis in $V$, and this independence is justified in Proposition~\ref{prop:wedge-indep-basis}.
\end{remark*}

\begin{remark*}
    In the remaining propositions of this subsection, we assume that $n$, $N$, $V$, and $\{e_S\}_{S\subseteq[n]}$ are as in Definition~\ref{def:exterior-algebra}.
\end{remark*}

The exterior product satisfies the following properties.

\begin{proposition}\label{prop:wedge-alt}
Let $v$ and $w$ be vectors in $V$. Then $v \wedge w = -w \wedge v$.
\end{proposition}
\begin{proof}
By bilinearity, it suffices to verify the statement for $v = e_{\fbr{a}}$, $w = e_{\fbr{b}}$. If $a = b$, both sides are zero. Otherwise,
\[
e_{\fbr{a}} \wedge e_{\fbr{b}} = \sgn(\fbr{a}, \fbr{b})\, e_{\fbr{a, b}} = -\sgn(\fbr{b}, \fbr{a})\, e_{\fbr{a, b}} = -e_{\fbr{b}} \wedge e_{\fbr{a}}.
\]
\end{proof}

\begin{proposition}\label{prop:wedge-associative}
    The exterior product on $\bigwedge V$ is associative.
\end{proposition}
\begin{proof}
By bilinearity, it suffices to check that for all $A, B, C \subseteq N$,
\[
(e_A \wedge e_B) \wedge e_C = e_A \wedge (e_B \wedge e_C).
\]
Assume that $A, B, C$ are pairwise disjoint, as the product is zero otherwise. Then,
\begin{align*}
(e_A \wedge e_B) \wedge e_C &= \sgn(A, B)\, \sgn(A \cup B, C)\, e_{A \cup B \cup C},\\
e_A \wedge (e_B \wedge e_C) &= \sgn(B, C)\, \sgn(A, B \cup C)\, e_{A \cup B \cup C}.
\end{align*}
Both expressions are equal since $\inv(A, B) + \inv(A \cup B, C) = \inv(A, B \cup C) + \inv(B, C)$.
\end{proof}

The following proposition, analogous to one from \cite{2019Rosen}, relates the exterior and scalar products.

\begin{proposition}\label{prop:wedge-to-sp}
    Let $k\ge 0$ and $l\ge 0$ be integers, and let $v_1, \ldots, v_k$ and $w_1, \ldots, w_l$ be vectors from $V$. If $l = k$, then
\[
\sp{v_1 \wedge \cdots \wedge v_k,\ w_1 \wedge \cdots \wedge w_k}
= \begin{vmatrix}
\sp{v_1, w_1} & \cdots & \sp{v_1, w_k} \\
\vdots & \ddots & \vdots \\
\sp{v_k, w_1} & \cdots & \sp{v_k, w_k}
\end{vmatrix}.
\]
If $l \ne k$, then
\[
\sp{v_1 \wedge \cdots \wedge v_k,\ w_1 \wedge \cdots \wedge w_l} = 0.
\]
\end{proposition}
\begin{proof}
For $l \ne k$, by multilinearity, it suffices to verify the equality when the $v_i$ and $w_j$ are basis vectors, which is straightforward.

For $l = k$, by multilinearity and Proposition~\ref{prop:wedge-alt}, it suffices to check the identity for $S, T \in \binom{N}{k}$, where $S = \fbr{s_1 < \ldots < s_k}\subseteq N$ and $T = \fbr{t_1 < \ldots < t_k}\subseteq N$. In this case, $\sp{e_{s_1} \wedge \cdots \wedge e_{s_k},\ e_{t_1} \wedge \cdots \wedge e_{t_k}} = \sp{e_S, e_T}$, which equals $1$ if $S = T$ and $0$ otherwise. The determinant on the right-hand side reflects exactly this.
\end{proof}

As a consequence, we have the following proposition.

\begin{proposition}\label{prop:wedge-indep-basis}
Let $\fbr{f_v}_{v \in N}$ be an orthonormal basis for $V$ (not necessarily equal to $\fbr{e_v}_{v \in N}$). For $S = \fbr{s_1 < \ldots < s_k} \subseteq N$, define $f_S := \bigwedge_{i\in[k]} f_{s_i}$. Then the collection $\fbr{f_S}_{S \subseteq N}$ is an orthonormal basis for $\bigwedge V$. Moreover, for all $S, T \subseteq N$,
\[
f_S \wedge f_T =
\begin{cases}
\sgn(S, T)\, f_{S \cup T}, &\text{if } S \cap T = \emptyset; \\
0, & \text{otherwise}.
\end{cases}
\]
\end{proposition}
\begin{proof}
By Proposition~\ref{prop:wedge-to-sp},
\[
\forall S, T \subseteq N \quad \sp{f_S, f_T} =
\begin{cases*}
1, & if $S = T$;\\
0, & if $S \ne T$.
\end{cases*}
\]
Hence, $\fbr{f_S}_{S \subseteq N}$ is an orthonormal basis.

Also, again by Proposition~\ref{prop:wedge-to-sp},
\[
\forall R, S, T \subseteq N \quad \sp{f_R, f_S \wedge f_T} =
\begin{cases}
    \sgn(S,\ T), &\text{if }S \cap T = \emptyset \text{ and } R=S\cup T;\\
    0, &\text{otherwise}.
    \end{cases}
\]
\end{proof}

Proposition~\ref{prop:wedge-indep-basis} shows that the exterior and scalar products on $\bigwedge V$ are independent of the choice of orthonormal basis of $V$. Therefore, we may refer to the exterior algebra over $V$ without specifying a basis.

\subsection{Generic Basis} \label{subsec:generic-basis}

The proof of Theorem~\ref{thrm:main-theorem} relies on the notion of a generic pair of bases.

\begin{definition}
Let $N$ be a finite set of size $n$ equipped with a linear order. Let $V$ be a real vector space of dimension $n$. A pair of orthonormal bases $(\fbr{e_v}_{v \in N}, \fbr{f_v}_{v \in N})$ for $V$ is called \emph{generic} if for all subsets $S \subseteq N$ and $T \subseteq N$ of the same size we have $\sp{f_S, e_T} \ne 0$.
\end{definition}

\begin{remark*}
This definition is symmetric: if the pair $(\fbr{e_v}_{v\in N}, \fbr{f_v}_{v\in N})$ is generic, then so is the pair $(\fbr{f_v}_{v\in N}, \fbr{e_v}_{v\in N})$. Therefore, the order is not matter in what follows.
\end{remark*}

\begin{theorem}\label{thrm:generic-basis-exists}
Let $V$ be a vector space with an orthonormal basis $\fbr{e_v}_{v \in N}$. Then there exists an orthonormal basis $\fbr{f_v}_{v \in N}$ such that the pair $(\fbr{e_v}_{v\in N}, \fbr{f_v}_{v\in N})$ is generic.
\end{theorem}

To prove this theorem, we use the following simple consequence of Proposition~\ref{prop:wedge-to-sp}.

\begin{lemma}\label{lemma:eqviv-to-be-generic}
Let $\fbr{f_v}_{v \in N}$ be a basis for $V$ with transition matrix $A = \fbr{a_{v,w}}_{v,w \in N}$, i.e., $f_v = \sum_{w \in N} a_{v,w} \cdot e_w$. Then for all $S, T \subseteq N$ of equal size,
\[
\sp{f_S, e_T} = \det(A_{S|T}),
\]
where $A_{S|T}$ denotes the submatrix of $A$ with rows indexed by $S$ and columns indexed by $T$.
\end{lemma}

Thus, Theorem~\ref{thrm:generic-basis-exists} reduces to the following lemma.

\begin{lemma}\label{lemma:generic-matrix-exists}
There exists a real orthonormal $n \times n$ matrix in which all square minors are nonzero.
\end{lemma}

This lemma was proven in \cite{1984KalaiPattern}, but we provide a more direct proof.

\begin{proof}
Let $R$ be any commutative ring. We define an orthogonalization procedure $F_R: \mathrm{Mat}_{n \times n}(R) \to \mathrm{Mat}_{n \times n}(R)$. For every matrix $A \in \mathrm{Mat}_{n \times n}(R)$, perform the following $\binom{n}{2}$ transformations: for $i$ from $1$ to $n$, and $j$ from $1$ to $i-1$, replace the $i$-th row $v_i$ with $\sp{v_j, v_j} v_i - \sp{v_i, v_j} v_j$, where $\sp{u, w} = \sum_{k=1}^n u_k w_k$. It is easy to verify that for every matrix $A$, the matrix $F_R(A)$ has orthogonal rows. Also, if $A$ is orthonormal, then $F_R(A) = A$.

For any ring homomorphism $f: R \to Q$ and the induced matrix homomorphism $f: \mathrm{Mat}_{m \times n}(R) \to \mathrm{Mat}_{m \times n}(Q)$, we have
\[
\forall v_1, v_2 \in \mathrm{Mat}_{1 \times n}(R) \quad
f\left(\sp{v_2, v_2} v_1 - \sp{v_1, v_2} v_2\right) = \sp{f(v_2), f(v_2)} f(v_1) - \sp{f(v_1), f(v_2)} f(v_2),
\]
and hence
\begin{equation}\label{eq:ring-exists-generic}
\forall A \in \mathrm{Mat}_{n \times n}(R) \quad f(F_R(A)) = F_Q(f(A)).
\end{equation}

Consider the matrix $A \in \mathrm{Mat}_{n \times n}(\Z[\fbr{X_{i,j}}_{i\in[n],j \in [n]}])$ where $A_{i,j} = X_{i,j}$. Let $B := F_{\Z[\fbr{X_{i,j}}]}(A)$. We will show that every square minor $\det(B_{S|T})$ is a nonzero polynomial with integer coefficients.

Let $S = \fbr{s_1 < \ldots < s_k}$, $T = \fbr{t_1 < \ldots < t_k}$ be subsets of $[n]$. Choose any permutation $\pi: [n] \to [n]$ such that $\pi(s_i) = t_i$ for all $i$. Define a ring homomorphism $f: \Z[\fbr{X_{i,j}}] \to \Z$ by
\[f(X_{i,j})=
\begin{cases}
1, & \text{if }\pi(i)=j;\\
0, & \text{otherwise}.
\end{cases}\]
Then $f(A)$ is a permutation matrix, which is orthonormal, so $F_{\Z}(f(A)) = f(A)$ and
\[
\det(F_{\Z}(f(A))_{S|T}) = \det(f(A)_{S|T}) = 1.
\]
By identity~\eqref{eq:ring-exists-generic}, $f(B) = F_{\Z}(f(A))$, so $\det(f(B)_{S|T}) \ne 0$ and thus $\det(B_{S|T})$ is a nonzero polynomial.

Now take $\fbr{a_{i,j}}_{i,j \in [n]}$ to be a collection of $n^2$ real numbers that are algebraically independent over $\Q$. Define a ring homomorphism $g: \Z[\fbr{X_{i,j}}] \to \R$ by $g(X_{i,j}) = a_{i,j}$. Let $C := g(B)$. Then $C$ has orthogonal rows, and by algebraic independence, all square minors of $C$ are nonzero. Since $\det(C) \ne 0$, we can normalize the rows to obtain an orthonormal matrix with all square minors nonzero.
\end{proof}

\subsection{Colorful Generic Basis} \label{subsec:colorful-generic-basis}

Since we want to work with hypergraphs whose vertices are colored, it is useful to introduce a coloring on the set $N$ as well.

\begin{definition}
Let $N$ be a linearly ordered set and $[d]$ a set of colors. A coloring $c: N \to [d]$ is said to be \emph{compatible} with the order on $N$ if for all $u \le v$, we have $c(u) \le c(v)$.

The set of elements of color $i$ is denoted by $N_i := c^{-1}(i)$.
\end{definition}

\begin{definition}
Let $N$ be a linearly ordered set, and let $c: N \to [d]$ be a coloring compatible with the order on $N$. A pair of orthonormal bases $(\fbr{e_v}_{v \in N}, \fbr{f_v}_{v \in N})$ for $V$ is called \emph{colorful} if for all $u, v \in N$ with $c(u) \ne c(v)$, we have $\sp{f_v, e_u} = 0$.
\end{definition}

\begin{remark*}
This definition is symmetric: if the pair $(\fbr{e_v}_{v\in N}, \fbr{f_v}_{v\in N})$ is colorful, then so is the pair $(\fbr{f_v}_{v\in N}, \fbr{e_v}_{v\in N})$. Therefore, the order is not matter in what follows.
\end{remark*}

The following lemma shows a key property of a colorful generic pair of orthonormal bases.

\begin{lemma}\label{lemma:e_in_terms_of_generic}
Let $V$ be a vector space, and let $c: N \to [d]$ be a coloring compatible with the order on $N$. Suppose we have a colorful pair of orthonormal bases $\fbr{e_v}_{v \in N}$ and $\fbr{f_v}_{v \in N}$. Then
\[
\forall R \subseteq N \quad e_R = \bigwedge_{i\in[d]}\ \sum_{S_i \in \binom{N_i}{|R_i|}} \sp{f_{S_i}, e_{R_i}} \cdot f_{S_i},
\]
where $R_i := R \cap N_i$.
\end{lemma}

\begin{proof}
Fix $R \subseteq N$, $R = \fbr{r_1 < \ldots < r_{|R|}}$.

We need to determine for which subsets $S \subseteq N$, with $S = \fbr{s_1 < \ldots < s_{|S|}}$, the inner product $\sp{f_S, e_R} \ne 0$. This requires $|S| = |R| = k$.

By Proposition~\ref{prop:wedge-to-sp}, $\sp{f_S, e_R}$ is equal to the determinant of the matrix
\[
\begin{vmatrix}
\sp{f_{s_1}, e_{r_1}} & \cdots & \sp{f_{s_1}, e_{r_{|R|}}} \\
\vdots & \ddots & \vdots \\
\sp{f_{s_{|S|}}, e_{r_1}} & \cdots & \sp{f_{s_{|S|}}, e_{r_{|R|}}}
\end{vmatrix}.
\]

Since the pair of bases is colorful and the coloring $c$ is compatible with the order on $N$, this matrix is block-diagonal with $d$ blocks. For each $i$, let $S_i := S \cap N_i = \fbr{s_{i,1} < \ldots < s_{i,|S_i|}}$ and $R_i := R \cap N_i = \fbr{r_{i,1} < \ldots < r_{i,|R_i|}}$. The $i$-th block is
\[
\begin{vmatrix}
\sp{f_{s_{i,1}}, e_{r_{i,1}}} & \cdots & \sp{f_{s_{i,1}}, e_{r_{i,|R_i|}}} \\
\vdots & \ddots & \vdots \\
\sp{f_{s_{i,|S_i|}}, e_{r_{i,1}}} & \cdots & \sp{f_{s_{i,|S_i|}}, e_{r_{i,|R_i|}}}
\end{vmatrix}.
\]

If for some $i$ we have $|S_i| < |R_i|$, then this determinant is zero. Since $|S| = |R|$, for all $i\in [d]$ we must have
$|S \cap N_i| = |R \cap N_i|$.
In this case,
\[
\sp{f_S, e_R} = \prod_{i\in[d]} \sp{f_{S \cap N_i}, e_{R \cap N_i}}.
\]
\end{proof}

We want all $\sp{f_{S_i}, e_{R_i}}$ to be nonzero for $|S_i| = |R_i|$, which motivates the following definition combining colorful and generic bases.

\begin{definition}
A pair of orthonormal bases $(\fbr{e_v}_{v \in N}, \fbr{f_v}_{v \in N})$ for $V$ is called \emph{colorful generic} if it is colorful and
\[
\forall i\in[d],\ \forall S, T \subseteq N_i,\ |S| = |T| \quad \sp{f_S, e_T} \ne 0.
\]
\end{definition}

A colorful generic pair of orthonormal bases exists, as the following theorem shows.

\begin{theorem}\label{thrm:colorful-generic-basis-exists}
Let $V$ be a vector space with an orthonormal basis $\fbr{e_v}_{v \in N}$, and let $c: N \to [d]$ be a coloring compatible with the order on $N$. Then there exists an orthonormal basis $\fbr{f_v}_{v \in N}$ such that the pair $(\fbr{e_v}_{v \in N}, \fbr{f_v}_{v \in N})$ is colorful generic.
\end{theorem}

\begin{proof}
For each $i$, by Theorem~\ref{thrm:generic-basis-exists}, there exists an orthonormal set of vectors $\fbr{f_v}_{v \in N_i}$ such that the pair $(\fbr{e_v}_{v \in N_i}, \fbr{f_v}_{v \in N_i})$ is generic in the subspace $\mathrm{span}(\fbr{e_v}_{v \in N_i})$.

It is easy to check that the union $\bigcup_{i \in [d]} \fbr{f_v}_{v \in N_i}$ satisfies the requirements of the theorem.
\end{proof}

The following proposition follows from Lemma~\ref{lemma:e_in_terms_of_generic}.

\begin{proposition}\label{prop:e_in_terms_of_generic}
Let $V$ be a vector space, let $c: N \to [d]$ be a coloring compatible with the order on $N$, and let $(\fbr{e_v}_{v \in N},\fbr{f_v}_{v \in N})$ be a colorful generic pair of orthonormal bases. Then for every $R \subseteq N$,
\begin{align*}
e_R &= \bigwedge_{i\in[d]}\ \sum_{S_i \in \binom{N_i}{|R_i|}} \lambda_{S_i} \cdot f_{S_i}, \\
f_R &= \bigwedge_{i\in[d]}\ \sum_{S_i \in \binom{N_i}{|R_i|}} \mu_{S_i} \cdot e_{S_i},
\end{align*}

where $R_i := R \cap N_i$ for each $i\in[d]$, with
$\lambda_{S_i} := \sp{e_{S_i}, f_{R_i}}$ and
$\mu_{S_i} := \sp{f_{S_i}, e_{R_i}}$
nonzero.
\end{proposition}

\subsection{Left Interior Product} \label{subsec:left-interior}

To construct the mapping used in Lemma~\ref{lemma:lower-bound-kalai}, we need the following bilinear operation on the exterior algebra.

\begin{definition}
Let $N$ be a finite set of size $n$ with a fixed linear order. Let $V$ be a real vector space of dimension $n$ with a fixed orthonormal basis $\fbr{e_v}_{v \in N}$. Define the bilinear operation called the \emph{left interior product} $\llcorner: \bigwedge V \times \bigwedge V \to \bigwedge V$, which acts on basis vectors as follows:
\[
\forall S, T \subseteq N \quad e_T \llcorner e_S =
\begin{cases}
\sgn(S \setminus T,\ T) \cdot e_{S \setminus T}, & \text{if } T \subseteq S; \\
0, & \text{otherwise}.
\end{cases}
\]
\end{definition}

To show that this operation does not depend on the choice of orthonormal basis, we prove the following lemma.

\begin{lemma}\label{lemma:llcorner-prop-for-indep}
For all elements $h, f, g \in \bigwedge V$, we have
\[
\sp{h,\ g \llcorner f} = \sp{h \wedge g,\ f}.
\]
\end{lemma}

\begin{proof}

By multilinearity, it suffices to verify the identity on basis vectors. Let $g=e_T$, $f=e_S$, and $h=e_R$. We may assume that $T\subseteq S$, since otherwise both sides equal zero. Likewise, we may assume that $R\subseteq S\setminus T$ and $|R|+|T|=|S|$, which forces $R=S\setminus T$. Therefore,
\[
\sp{e_{S \setminus T},\, e_T \llcorner e_S}
= \sgn(S \setminus T,\, T)
= \sp{e_{S \setminus T} \wedge e_T,\, e_S}.
\]
\end{proof}

\begin{proposition}\label{prop:llcorner-basis-formula}
Let $\fbr{f_v}_{v \in N}$ be an orthonormal basis for $V$. Then for all $S, T \subseteq N$,
\[
f_T \llcorner f_S =
\begin{cases}
\sgn(S \setminus T,\ T) \cdot f_{S \setminus T}, & \text{if } T \subseteq S; \\
0, & \text{otherwise}.
\end{cases}
\]
\end{proposition}

\begin{proof}
Let $R \subseteq N$ be arbitrary. By Lemma~\ref{lemma:llcorner-prop-for-indep},
\[
\sp{f_R,\ f_T \llcorner f_S} = \sp{f_R \wedge f_T,\ f_S} =
\begin{cases}
\sgn(S \setminus T,\ T), & \text{if } T \subseteq S,\ R = S \setminus T; \\
0, & \text{otherwise}.
\end{cases}
\]
\end{proof}

Thus, the definition of $\llcorner$ is independent of the choice of an orthonormal basis in $V$.

The following lemma shows how $\llcorner$ interacts with a colorful generic basis.

\begin{lemma}\label{lemma:llcorner-with-wedge}
Let $N$ be a finite set of size $n$ with a linear order. Let $c: N \to [d]$ be a coloring compatible with the order on $N$. Let $V$ be a real vector space of dimension $n$ with orthonormal basis $\fbr{e_v}_{v \in N}$. Then for all $S, T \subseteq N$,
\[
e_T \llcorner e_S = (-1)^q \bigwedge_{i\in[d]} e_{T_i} \llcorner e_{S_i},
\]
where $T_i = T \cap N_i$, $S_i = S \cap N_i$, and $q = \sum_{i\in[d]} \sum_{j=i+1}^{d} (|S_j| - |T_j|) |T_i|$.
\end{lemma}

\begin{proof}
The proof follows similarly to Lemma 3.4 in \cite{2023Bulavka}.

We may assume that $T_i \subseteq S_i$, since otherwise both sides equals zero.

We have:
\[
(-1)^q \bigwedge_{i\in[d]} e_{T_i} \llcorner e_{S_i} = (-1)^q \left(\prod_{i\in[d]} \sgn(S_i \setminus T_i,\ T_i)\right) e_{S \setminus T}.
\]
It suffices to show:
\[
\inv(S \setminus T,\ T) = q + \sum_{i\in[d]} \inv(S_i \setminus T_i,\ T_i).
\]

Since the coloring $c$ is compatible with the order on $N$:
\begin{multline*}\inv(S\setminus T,T) = |\fbr{s\in S\setminus T,t\in T\mid s > t}| = \sum_{i\in[d]}\sum_{j=i}^{d}|\fbr{s\in S_j\setminus T_j,t\in T_i\mid s > t}|\\
    = \sum_{i\in[d]}\sum_{j=i}^{d}(|S_j|-|T_j|)|T_i|=q + \sum_{i\in[d]}(|S_i|-|T_i|)|T_i|=q + \sum_{i\in[d]} \inv(S_i\setminus T_i, T_i).
\end{multline*}
\end{proof}

The following result relates the left interior product to a colorful generic pair of orthonormal bases.

\begin{proposition}\label{prop:llcorner-give-all-edges}
Let $V$ be a real vector space, and let $c: N \to [d]$ be a coloring compatible with the order on $N$. Let $(\fbr{e_v}_{v \in N}, \fbr{f_v}_{v \in N})$ be a colorful generic pair of orthonormal bases. Then for all $R, S \subseteq N$,
\[
f_R \llcorner e_S = \bigwedge_{i\in[d]} \sum_{W_i \in \binom{S_i}{|S_i| - |R_i|}} \lambda_{W_i} \cdot e_{W_i},
\]
where $R_i = R \cap N_i$, $S_i = S \cap N_i$, and $\lambda_{W_i}$ are nonzero scalars.
\end{proposition}

\begin{proof}
By Proposition~\ref{prop:e_in_terms_of_generic},
\[
f_R = \bigwedge_{i\in[d]} \sum_{W_i \in \binom{N_i}{|R_i|}} \mu_{W_i} \cdot e_{W_i}.
\]
By multilinearity of the left interior product and Lemma~\ref{lemma:llcorner-with-wedge},
\[
f_R \llcorner e_S = (-1)^q \bigwedge_{i\in[d]} \sum_{W_i \in \binom{N_i}{|R_i|}} \mu_{W_i} \cdot e_{W_i} \llcorner e_{S_i},
\]
where \[q = \sum_{i\in[d]} \sum_{j=i+1}^d |W_i| (|S_j| - |W_j|) = \sum_{i\in[d]} \sum_{j=i+1}^d |R_i| (|S_j| - |R_j|).\]

The term $e_{W_i} \llcorner e_{S_i}$ is nonzero only when $W_i \subseteq S_i$. Thus:
\[
f_R \llcorner e_S = (-1)^q \bigwedge_{i\in[d]} \sum_{T_i \in \binom{S_i}{|S_i| - |R_i|}} \mu_{S_i \setminus T_i} \cdot \sgn(T_i,\ S_i \setminus T_i) \cdot e_{T_i},
\]
and the result follows with the following coefficients:
\[
\lambda_{T_i} = \mu_{S_i \setminus T_i} \cdot \sgn(T_i,\ S_i \setminus T_i) \cdot
\begin{cases}
(-1)^q, & i = 1; \\
1, & i \ne 1.
\end{cases}
\]
\end{proof}

The following proposition is used in the proof of Theorem~\ref{thrm:main-theorem} to control how coefficients of constructed vectors depend on their parameters.

\begin{proposition}\label{prop:llcorner-sign-decompose}
Let $N$ be a finite set of size $n$ with a linear order, and let $c: N \to [d]$ be a coloring compatible with the order on $N$. Let $V$ be a real vector space of dimension $n$ with orthonormal basis $\fbr{e_v}_{v \in N}$. Identify $N_i$ with $[n_i] \times \fbr{i}$, where $n_i = |N_i|$.

Let $\mb{r}, \mb{s}, \mb{m} \in \Nz^d$ be integer vectors such that $\forall i \in [d],\ r_i \ge s_i \ge m_i$. Let
\[
T = \bigdisjointcup_{i \in [d]} T_i \times \fbr{i} \subseteq N,\quad \forall i \in [d]\quad T_i \in \binom{[m_i - 1] \cup ([n_i] \setminus [r_i])}{m_i}.
\]
Then,
\[
\bigwedge_{i\in[d]} e_{([r_i] \setminus [s_i]) \times \fbr{i}} \ \llcorner\ \bigwedge_{i\in[d]} e_{(T_i \cup ([r_i] \setminus [m_i])) \times \fbr{i}} = (-1)^{q_1(\mb{r}, \mb{s}) + q_2(\mb{r}, T) + q_3(\mb{s}, T)} \bigwedge_{i\in[d]} e_{(T_i \cup ([s_i] \setminus [m_i])) \times \fbr{i}},
\]
where:
\begin{align*}
q_1(\mb{r}, \mb{s}) &= \sum_{i \in [d]} r_i \sum_{j \in [d] \setminus \fbr{i}} s_j, \\
q_2(\mb{r}, T) &= \sum_{i \in [d]} r_i\cdot|T_i \setminus [m_i - 1]|, \\
q_3(\mb{s}, T) &= -\sum_{i \in [d]} s_i\cdot\left( |T_i \setminus [m_i - 1]| + \sum_{j \in [d] \setminus \fbr{i}} s_j \right).
\end{align*}
\end{proposition}

\begin{proof}
By Lemma~\ref{lemma:llcorner-with-wedge} and Proposition~\ref{prop:llcorner-basis-formula},
\begin{align*}
\bigwedge_{i\in[d]} e_{([r_i] \setminus [s_i]) \times \fbr{i}} \ \llcorner\ \bigwedge_{i\in[d]} e_{(T_i \cup ([r_i] \setminus [m_i])) \times \fbr{i}} &= (-1)^{q} \bigwedge_{i\in[d]} e_{([r_i] \setminus [s_i]) \times \fbr{i}}\ \llcorner\ e_{(T_i \cup ([r_i] \setminus [m_i])) \times \fbr{i}}\\
&=(-1)^{\tilde{q}} \bigwedge_{i\in[d]} e_{(T_i \cup ([s_i] \setminus [m_i])) \times \fbr{i}},
\end{align*}
where:
\[
\tilde{q} = \sum_{i \in [d]} (r_i - s_i) \left( \sum_{j \in [d] \setminus \fbr{i}} s_j \right) + \sum_{i \in [d]} (r_i - s_i)\cdot|T_i \setminus [m_i - 1]|,
\]
since \[\inv(T_i\disjointcup[s_i]\setminus[m_i],\ [r_i]\setminus[s_i])=(r_i-s_i)\cdot|T_i\setminus[r_i]|=(r_i-s_i)\cdot|T_i\setminus[m_i - 1]|.\]
This yields:
\[
\tilde{q} = q_1(\mb{r}, \mb{s}) + q_2(\mb{r}, T) + q_3(\mb{s}, T).
\]
\end{proof}

\subsection{Lower Bound via exterior algebra} \label{subsec:lower-bound-via-ext}

\begin{definition}
Let $G$ be a hypergraph. Take a real vector space $V$ with an orthonormal basis $\fbr{e_v}_{v \in V(G)}$ and define \[\spann G := \spann \fbr{e_S \mid S \in E(G)} \subseteq \bigwedge V.\]

For an element $m \in \spann G$, define its support as
\[
\supp(m) := \fbr{S \subseteq V(G) \mid \sp{e_S,\ m} \ne 0}.
\]
\end{definition}

The following lemma reduces the task of obtaining a lower bound for $\wdir((G,c),(H,t))$ to finding a suitable linear subspace of $\spann G$. It is analogous to Lemma 4.1 from \cite{2023Bulavka}.

\begin{lemma}\label{lemma:lower-bound-wdir-via-kernel}
Let $d \ge 1$ be an integer, let $(G, c)$ be a $d$-colored hypergraph, and let $\mc{H}$ be a family of $d$-colored hypergraphs. Let $U \subseteq \spann G$ be a subspace such that for every colored copy $(\wt{H},\tilde{t})$ of a hypergraph $(H, t) \in \mc{H}$ in $G$, there exists an element $m \in U$ with $\supp(m) = E(\wt{H})$. Then
\[
\wdir((G, c), \mc{H}) \ge |E(G)| - \dim U.
\]
\end{lemma}

\begin{proof}
From linear algebra, there exists a linear subspace $Y$ such that $\spann G = Y \oplus U$. Let $\Gamma: E(G) \to Y$ be the projection onto $Y$. Then $\Gamma$ satisfies the conditions of Lemma~\ref{lemma:lower-bound-kalai}, and hence
\[
\wdir((G, c), \mc{H}) \ge \rk(\Gamma(E(G))) = |E(G)| - \dim U.
\]
\end{proof}

\section{Lower Bound in Theorem \ref{thrm:main-theorem}} \label{sec:lower-bound-main-thrm}

\begin{proof}

Without loss of generality, assume that $\mc{R} = \mc{R}(\mb{n}) \neq \emptyset$.

Let $N = N_1 \disjointcup \ldots \disjointcup N_d$, where for each $i \in [d]$, we define $N_i = [n_i] \times \{i\}$. Define a coloring $c: N \to [d]$ such that $c(N_i) = \{i\}$. Let introduce a linear order on $N$ by declaring that for all $(a,i), (b,j) \in N$, we have $(a,i) < (b,j)$ if and only if $(i < j)$ or $(i = j$ and $a < b)$.


Let $V$ be a vector space of dimension $\sum_{i \in [d]} n_i$. Take a colorful generic pair of orthonormal bases $\fbr{e_v}_{v \in N}, \fbr{f_v}_{v \in N}$, which exists by Theorem \ref{thrm:colorful-generic-basis-exists}.

For each $\mb{r} \in \mc{R}$ define:
\[
g_{\mb{r}} = \sum_{\mb{s} \in \mc{S}} (-1)^{q_1(\mb{r}, \mb{s})} \bigwedge_{i\in[d]} f_{([r_i] \setminus [s_i]) \times \{i\}},
\]
where $q_1(\mb{r}, \mb{s}) = \sum_{i \in [d]} r_i \sum_{j \in [d] \setminus \{i\}} s_j$ (as in Proposition \ref{prop:llcorner-sign-decompose}).

Define
\[
U = \spann \{g_{\mb{r}} \llcorner e_R \mid \mb{r}\in\mc{R},\ R \subseteq N, \text{ such that}\ |R \cap N_i| = r_i  \text{ for all } i\in[d]\}.
\]

We now verify that $U$ satisfies the conditions of Lemma \ref{lemma:lower-bound-wdir-via-kernel}.

Fix $\mb{r} \in \mc{R}$ and consider a colored copy $H$ of the hypergraph $\tprod{\mc{S}}{\mb{r}}$ in $\tprod{\mc{S}}{\mb{n}}$. Define $m = g_{\mb{r}} \llcorner e_{V(H)}$, then $m \in U$ and by Proposition \ref{prop:llcorner-give-all-edges},
\[
m = \sum_{\mb{s} \in \mc{S}} (-1)^{q_1(\mb{r}, \mb{s})} \bigwedge_{i\in[d]} \sum_{W_i \in \binom{V(H) \cap N_i}{s_i}} \lambda_{W_i} \cdot e_{W_i \times \{i\}} = \sum_{S \in E(H)} \mu_S \cdot e_S,
\]
where each $\mu_S$ is a nonzero constant. Thus, $\supp(m) = E(H)$ and the conditions of the Lemma \ref{lemma:lower-bound-wdir-via-kernel} are satisfied.

It remains to determine $\dim U$.

By Proposition \ref{prop:e_in_terms_of_generic},
\[
U \subseteq \spann \{g_{\mb{r}} \llcorner f_R \mid \mb{r}\in\mc{R},\ R \subseteq N, \text{ such that}\ |R \cap N_i| = r_i \text{ for all } i\in[d]\}.
\]

For each such $R$ and every $i\in [d]$, there exists a minimal $m_i \le r_i$ such that $(m_i,i) \notin R \cap N_i$ and $([r_i] \setminus [m_i])\times\fbr{i} \subseteq R \cap N_i$. Hence,
\begin{multline*}
U  \subseteq \spann \bigcup_{\mb{m} \in \Nz^d} \big\{g_{\mb{r}} \llcorner f_R\ \big|\ \mb{r}\in\mc{R},\ R \subseteq N, \text{ such that }\\ |R \cap N_i| = r_i,\ ([r_i] \setminus [m_i])\times\fbr{i} \subseteq R \cap N_i,\ (m_i,i) \notin R \cap N_i \text{ for all } i\in[d]\}.
\end{multline*}

If for some $\mb{s} \in \mc{S}$ and $i \in [d]$, $m_i > s_i$, then $([r_i] \setminus [s_i])\times\fbr{i} \not\subseteq R\cap N_i$ and
\[
\bigwedge_{i\in[d]} f_{([r_i] \setminus [s_i]) \times \{i\}} \llcorner f_R = 0.
\]
Thus, we can restrict to $\mb{m} \in \downcl{\mc{S}}$, and we get:
\begin{multline*}
    \dim U \le \sum_{\mb{m}\in \downcl{\mc{S}}}\rk \bigcup_{\mb{r}\in\mc{R}}\Bigg\{\sum_{\substack{\mb{s}\in\mc{S}\\\forall i\in [d]\ m_i\le s_i}} (-1)^{q_1(\mb{r},\mb{s})}\bigwedge_{i\in[d]} f_{([r_i]\setminus[s_i])\times\fbr{i}}\ \llcorner\ \bigwedge_{i\in[d]} f_{(T_i\disjointcup ([r_i]\setminus[m_i]))\times\fbr{i}}\ \Bigg|\\
    \forall i\in[d]\ T_i\in \binom{[m_i-1]\disjointcup ([n_i]\setminus[r_i])}{m_i}\Bigg\}.
    \end{multline*}

By Proposition \ref{prop:llcorner-sign-decompose},
\begin{multline*}
    \dim U \le \sum_{\mb{m} \in \downcl{\mc{S}}} \rk \bigcup_{\mb{r} \in \mc{R}} \Bigg\{\sum_{\substack{\mb{s} \in \mc{S} \\ \forall i \in [d],\ m_i \le s_i}}
    (-1)^{2q_1(\mb{r}, \mb{s})+ q_2(\mb{r}, T) + q_3(\mb{s}, T)} \bigwedge_{i\in[d]} f_{(T_i \disjointcup ([s_i]\setminus[m_i])) \times \{i\}} \ \Bigg| \\
    \forall i \in [d]\ T_i \in \binom{[m_i - 1] \disjointcup ([n_i] \setminus [r_i])}{m_i}
    \Bigg\}.
\end{multline*}

Multiplying each expression for fixed $T$ by $(-1)^{q_2(\mb{r}, T)}$ (which does not change the rank), we obtain:
\begin{multline*}
\dim U \le \sum_{\mb{m} \in \downcl{\mc{S}}} \rk \bigcup_{\mb{r} \in \mc{R}} \Bigg\{
\sum_{\substack{\mb{s} \in \mc{S} \\ \forall i \in [d],\ m_i \le s_i}}
(-1)^{q_3(\mb{s}, T)} \bigwedge_{i\in[d]} f_{(T_i \disjointcup ([s_i]\setminus[m_i])) \times \{i\}} \ \Bigg| \\
\forall i \in [d]\ T_i \in \binom{[m_i - 1] \disjointcup ([n_i] \setminus [r_i])}{m_i}
\Bigg\}.
\end{multline*}

If, for a fixed $T$, there are several vectors $\mb{r}\in\mc{R}$ such that
\[
\forall i\in[d] \quad T_i \in \binom{[m_i-1]\disjointcup ([n_i]\setminus [r_i])}{m_i},
\]
then the resulting expression is the same for every such choice of $\mb{r}$. This is precisely why we need the detailed formula in Proposition~\ref{prop:llcorner-sign-decompose}: it allows us to eliminate the dependence on the choice of $\mb{r}$ for a given $T$. Thus,
\begin{equation*}
\dim U \le \sum_{\mb{m} \in \downcl{\mc{S}}} \left| \left\{
T \in \timesprod_{i \in [d]} \binom{[n_i]}{m_i} \ \middle| \
\exists \mb{r} \in \mc{R}\ \forall i \in [d]\ T_i \in \binom{[m_i - 1] \disjointcup ([n_i] \setminus [r_i])}{m_i}
\right\} \right|.
\end{equation*}

By the inclusion-exclusion principle, we obtain $\dim U \le q(\mb{n}, \mc{S}, \mc{R})$.

\end{proof}

\begin{remark*}
In fact, it can be shown that $\dim U = q(\mb{n}, \mc{S}, \mc{R})$, but we do not need this result in what follows.
\end{remark*}

\section{Upper Bound in Theorem \ref{thrm:main-theorem}}\label{sec:upper-bound-main-thrm}

\subsection{Case: $\mc{R}$ has a minimum element}\label{subsec:upper-bound-smallest-r}

\begin{proof}
We assume $\mc{R}(\mb{n}) \neq \emptyset$. Suppose $\tilde{\mb{r}} \in \mc{R}$ is such that for all $\mb{r} \in \mc{R}$ and all $i \in [d]$ we have $r_i \ge \tilde{r}_i$. Then $\tilde{\mb{r}} \in \mc{R}(\mb{n})$ and
\[
q(\mb{n}, \mc{S}, \mc{R}) = \sum_{\mb{m} \in \downcl \mc{S}} \ \prod_{\substack{i = 1 \\ m_i \neq 0}}^{d} \binom{m_i - 1 + n_i - \tilde{r}_i}{m_i}.
\]

For each $\mb{m} \in \downcl \mc{S}$ choose an arbitrary vector $\mb{s}_{\mb{m}} \in \mc{S}$ such that $s_{\mb{m},i} \ge m_i$ for all $i\in[d]$. Define
\[F=\tprod{\mc{S}}{\mb{n}}\ \Big\backslash\ \bigdisjointcup_{\mb{m}\in\downcl{\mc{S}}}\fbr{\bigdisjointcup_{i\in[d]}\ (T_i\disjointcup [s_{\mb{m},i} - m_i])\times \fbr{i}\ \middle|\ \forall i\in[d]\ T_i\in\binom{[n_i]\setminus[\tilde{r}_i-m_i+1]}{m_i}}.\]

We never remove the same edge twice, because from a removed edge $W$ one can uniquely recover the tuple $(\mb{m}, \fbr{T_{W,i}})$ as follows: in each part $W_i\subseteq [n_i]$ find the maximal prefix $[l_i]\times\{i\} \subseteq W_i$, then set $T_{W,i} = W_i \setminus [l_i]$ and $m_i = |T_{W,i}|$. Hence the size of the hypergraph $F$ matches the desired bound in the theorem.

We associate to each removed edge $W$ the tuple of tuples $T_W = (T_{W,1}, \dots, T_{W,d})$. Within each $T_{W,i}$ we order elements in decreasing order. We process the edges $W$ in lexicographic increasing order of these tuples; that is, $T_{W'} < T_W$ if there exists $i$ such that $T_{W',i} < T_{W,i}$ and for all $j < i$, $T_{W',j} = T_{W,j}$.

Suppose at the current step we are adding edge $W$, and let $m_i = |T_{W,i}|$. Construct a copy $H$ of the hypergraph $\tprod{\mc{S}}{\tilde{\mb{r}}}$ on vertex set $R = \bigdisjointcup_{i\in[d]} R_i \times \{i\}$, where $R_i = [\tilde{r}_i - m_i] \disjointcup T_{W,i}$. Suppose that another edge $W' \neq W$ of $H$ is missing. Let $m'_i = |T_{W',i}|$. Since $T_{W'} > T_W$, there is some $i$ with $T_{W',i} > T_{W,i}$. Then $T_{W',i}$ must completely contain $T_{W,i}$ (because the elements in each $T_{W,i}$ are sorted descending and each is larger than any in $R_i \setminus T_{W,i}$). Hence $m'_i > m_i$, and $T_{W',i}$ contains a vertex from $[\tilde{r}_i - m_i - (m'_i - m_i - 1)]$ that contradicts $T_{W',i} \subseteq [n_i] \setminus [\tilde{r}_i - m'_i + 1]$. Thus all edges in $E(H)$ except $W$ are already added, and the addition of $W$ completes a new colored copy of $\tprod{\mc{S}}{\tilde{\mb{r}}}$.
\end{proof}

\subsection{Case: $\mc{S}$ has a maximum element}\label{subsec:upper-bound-largest-s}

\begin{proof}
Suppose $\tilde{\mb{s}} \in \mc{S}$ is such that for all $\mb{s} \in \mc{S}$ and all $i$ we have $s_i \le \tilde{s}_i$. Then $\downcl \mc{S} = \timesprod_{i\in[d]} \bigl( \{0\} \disjointcup [\tilde{s}_i] \bigr)$ and hence
\begin{align*}
q(\mb{n}, \mc{S}, \mc{R}) &= \sum_{\emptyset \neq \mc{Q} \subseteq \mc{R}(\mb{n})} (-1)^{|\mc{Q}|+1} \prod_{i\in [d]} \left(1 + \sum_{m_i = 1}^{\tilde{s}_i} \binom{m_i - 1 + n_i - \max_{\mb{r} \in \mc{Q}} r_i}{m_i} \right) \\
&= \sum_{\emptyset \neq \mc{Q} \subseteq \mc{R}(\mb{n})} (-1)^{|\mc{Q}|+1} \prod_{i\in [d]} \binom{n_i - \max_{\mb{r} \in \mc{Q}} r_i + \tilde{s}_i}{\tilde{s}_i}.
\end{align*}

Define

\[F=\tprod{\mc{S}}{\mb{n}}\ \Big\backslash\ \bigcup_{\mb{r} \in \mc{R}(\mb{n})}\fbr{\bigdisjointcup_{i\in[d]}\ W_i\times \fbr{i}\ \middle|\ \forall i\in[d]\ W_i\in\binom{[n_i] \setminus [r_i - \tilde{s}_i]}{\tilde{s}_i}}.\]

The size of $F$ matches the desired bound in the theorem by the inclusion--exclusion principle.

For each missing edge $W = \bigdisjointcup_{i\in[d]}\ W_i\times \fbr{i}$, order each part $W_i$ in decreasing order and associate to $W$ a tuple $(W_1, \ldots, W_d)$. We process the edges $W$ in lexicographic increasing order of these tuples.

Suppose at the current step we are adding edge $W$. By definition of $F$, there exists $\mb{r} \in \mc{R}(\mb{n})$ such that $W_i \in \binom{[n_i] \setminus [r_i - \tilde{s}_i]}{\tilde{s}_i}$ for all $i$. Construct a copy $H$ of $\tprod{\mc{S}}{\mb{r}}$ on $R = \bigdisjointcup_{i\in[d]} R_i$, where $R_i = ([r_i - \tilde{s}_i] \disjointcup W_i) \times \{i\}$. Suppose that another edge $W' \neq W$ of $H$ is missing. By construction of $F$, $|W'_i| = \tilde{s}_i$ for all $i$, so $W'_i \le W_i$, contradicting the lexicographic order in which we add missing edges. Thus all edges in $E(H)$ except $W$ are already added, and the addition of $W$ completes a new colored copy of $\tprod{\mc{S}}{\mb{r}}$.
\end{proof}

\section{Corollaries for uncolored weak saturation}\label{sec:wsat-eq-wdir}

To reduce uncolored weak saturation to colored weak saturation, we represent an uncolored copy of the hypergraph $\tprod{\mc{S}}{\mb{r}}$ as a colored copy of a hypergraph from a suitable family $\mc{H}$. However, in such an uncolored copy of $\tprod{\mc{S}}{\mb{r}}$, several parts may be mapped to a single part of the host hypergraph. To handle this situation, we introduce the following definitions.

\begin{definition}
For a function $f: [d] \to [d]$ and integer vector $\mb{v} \in \Nz^d$, define the convolution $\conv{\mb{v}}{f} \in \Nz^d$ by
\[
\forall i \in [d]\quad(\conv{\mb{v}}{f})_i = \sum_{j \in f^{-1}(i)} v_j.
\]

For a family $\mc{R} \subseteq \Nz^d$ and a family $\mc{F}$ of functions from $[d]$ to $[d]$, set
\[
\fconv{\mc{R}}{\mc{F}} = \{\conv{\mb{r}}{f} \mid \mb{r} \in \mc{R}, f \in \mc{F}\}.
\]
\end{definition}

\begin{remark*}
    If $f$ is a permutation, then $\conv{\mb{v}}{f} = (v_{f^{-1}(1)}, \ldots, v_{f^{-1}(d)})$.
\end{remark*}

\begin{definition}
Let $\mc{S} \subseteq \Nz^d$ be a non-empty finite family of integer vectors. Define the convolution-compatible functions with respect to $\mc{S}$ by
\[
\Conv(\mc{S}) = \{ f: [d] \to [d] \mid \forall \mb{s} \in \mc{S},\ \conv{\mb{s}}{f} \in \mc{S} \}.
\]
\end{definition}


\begin{proposition}\label{prop:wsat-upper-via-conv-wdir}
Let $d$, $\mb{n}$, $\mc{S}$, and $\mc{R}$ be as in Theorem \ref{thrm:main-theorem}. Then
\[
\wsat(\tprod{\mc{S}}{\mb{n}},\ \tprod{\mc{S}}{\mc{R}}) \le \wdir\bigl(\tprod{\mc{S}}{\mb{n}},\, \tprod{\mc{S}}{\fconv{\mc{R}}{\Conv(\mc{S}) \cap S_d}}\bigr),
\]
where $S_d$ denotes the set of all permutations of $[d]$.
\end{proposition}

\begin{proof}
Let $H$ be an arbitrary colored copy of the hypergraph $\tprod{\mc{S}}{\conv{\mb{r}}{f}}$ in $\tprod{\mc{S}}{\mb{n}}$, for some $\mb{r} \in \mc{R}$ and $f \in \Conv(\mc{S}) \cap S_d$. It suffices to show that $H$ is also a uncolored copy of $\tprod{\mc{S}}{\mb{r}}$.

Let $V(H) = R = \bigdisjointcup_{i \in [d]} R_i$. We will show that $H$ is equal to the copy $\wt{H}$ of $\tprod{\mc{S}}{\mb{r}}$ with parts $\wt{R}_i = R_{f(i)}$. This choice of parts is possible because $|\wt{R}_i| = r_i$ for every $i\in[d]$.

Take any edge $W \in E(H)$. Then for some $\mb{s} \in \mc{S}$ we have $\forall i\in[d],\ |W \cap R_i| = s_i$. Hence
\[
\forall i\in[d]\quad |W \cap \wt{R}_i| = s_{f(i)} = (\conv{\mb{s}}{f^{-1}})_i.
\]
Since $f\in S_d$, for distinct $\mb{s}, \tilde{\mb{s}} \in \mc{S}$, $\conv{\mb{s}}{f} \neq \conv{\tilde{\mb{s}}}{f}$. Since $\mc{S}$ is finite and $f\in \Conv(\mc{S})$, for each $\mb{s} \in \mc{S}$ there exists $\tilde{\mb{s}} \in \mc{S}$ such that $\conv{\tilde{\mb{s}}}{f} = \mb{s}$. Thus $\forall i\in[d],\ |W \cap \wt{R}_i| = \tilde{s}_i$. So $E(H)\subseteq E(\wt{H})$.

Conversely, take an edge $W \in E(\wt{H})$. Then for some $\mb{s} \in \mc{S}$ we have $\forall i\in[d],\ |W \cap \wt{R}_i| = s_i$. That implies $\forall i\in[d],\ |W \cap R_i| = s_{f^{-1}(i)} = (\conv{\mb{s}}{f})_i$, and since $\conv{\mb{s}}{f} \in \mc{S}$, $W \in E(H)$. Thus $E(H) = E(\wt{H})$.
\end{proof}

The following theorem provides a broad range of cases in which uncolored weak saturation can be reduced to colored weak saturation.
\begin{theorem}\label{thrm:wsat-equal-wdir-general}
Let $d$, $\mb{n}$, $\mc{S}$, and $\mc{R}$ be as in Theorem \ref{thrm:main-theorem}. Assume:
\begin{enumerate}
    \item \label{item2:wsat-to-cwsat} $\Conv(\mc{S}) \subseteq S_d$.
    \item \label{item3:wsat-to-cwsat} $\forall i \in [d], \forall \mb{s}, \mb{s}' \in \mc{S}: |s_i - s'_i| \neq 1$.
    \item \label{item4:wsat-to-cwsat} $\forall i \in [d], \forall \mb{r} \in \mc{R}, \exists \mb{s} \in \mc{S}:\ s_i \neq 0 \text{ and } r_i \ge s_i + 1$.
\end{enumerate}
Then
\[
\wsat(\tprod{\mc{S}}{\mb{n}}, \tprod{\mc{S}}{\mc{R}}) = \wdir\bigl(\tprod{\mc{S}}{\mb{n}},\, \tprod{\mc{S}}{\fconv{\mc{R}}{\Conv(\mc{S})}}\bigr).
\]
\end{theorem}

\begin{proof}

Let $\mb{r} \in \mc{R}$ and suppose there is a uncolored copy $H$ of $\tprod{\mc{S}}{\mb{r}}$ in $\tprod{\mc{S}}{\mb{n}}$. We will show that it is a colored copy of $\tprod{\mc{S}}{\conv{\mb{r}}{f}}$ for some $f \in \Conv(\mc{S})$.

Let $V(H) = R = \bigdisjointcup_{i \in [d]} R_i$, where $\fbr{R_i}_{i\in[d]}$ are the preimages of parts in $\tprod{\mc{S}}{\mb{r}}$. Suppose that for some $i, j \in [d]$, $R_i \cap N_j \neq \emptyset$. Then $R_i \subseteq N_j$, otherwise pick $x, y \in R_i$ with $x \in N_j$ and $y \notin N_j$. By condition \ref{item4:wsat-to-cwsat}, there exists some $\mb{s} \in \mc{S}$ with $r_i > s_i \ge 1$. Therefore, we can choose an edge $W \in E(H)$ such that $x \in W_i$ and $y \notin W_i$. Let $W' = (W \setminus \{x\}) \cup \{y\}$. Then $W' \in E(H)$, but $|W \cap N_j| - |W' \cap N_j| = 1$, contradicting \ref{item3:wsat-to-cwsat}.

Hence there is a function $f: [d] \to [d]$ such that $\forall i\in[d]$ we have $R_i \subseteq N_{f(i)}$. If $f \notin \Conv(\mc{S})$, then $\exists \mb{s} \in \mc{S}$ such that $\conv{\mb{s}}{f} \notin \mc{S}$. Choose $W \in E(H)$ with $|W_i| = s_i$. Then $\forall i\in[d],\ |W \cap N_i| = (\conv{\mb{s}}{f})_i$ that contradicts $\conv{\mb{s}}{f} \notin \mc{S}$. Thus $f \in \Conv(\mc{S}) \subseteq S_d$, and $H$ is a colored‑copy of $\tprod{\mc{S}}{\conv{\mb{r}}{f}} \in \tprod{\mc{S}}{\fconv{\mc{R}}{\Conv(\mc{S})}}$.

Therefore $\wsat(\tprod{\mc{S}}{\mb{n}}, \tprod{\mc{S}}{\mc{R}}) \ge \wdir(\tprod{\mc{S}}{\mb{n}},\, \tprod{\mc{S}}{\fconv{\mc{R}}{\Conv(\mc{S})}})$. Requirement \ref{item2:wsat-to-cwsat} and Proposition \ref{prop:wsat-upper-via-conv-wdir} yield the equality.

\end{proof}

The requirement \ref{item4:wsat-to-cwsat} in Theorem \ref{thrm:wsat-equal-wdir-general} cannot be removed in general, as the following example illustrates.

\begin{example}
Let $d = 2$, $\mb{s} = (2,1)$, $\mb{r} = (2,2)$. For every $\mb{n} \in \Nz^2$ with $n_1 \ge 2, n_2 \ge 1$, one computes $q(\mb{n}, \{ \mb{s} \}, \{\mb{r}\}) = \binom{n_1}{2} (n_2 - 1)$, and by Theorem \ref{thrm:main-theorem},
\[
\wdir(\tprod{\mb{s}}{\mb{n}}, \tprod{\mb{s}}{\mb{r}}) = \binom{n_1}{2}.
\]
However, one can show
\[
\wsat(\tprod{\mb{s}}{\mb{n}}, \tprod{\mb{s}}{\mb{r}}) \le n_2
\]
by using a colored copies of $2$-colored hypergraph $H$ with parts $R_1 = \{1,2,3\}, R_2 = \{4\}$ and edges $E(H) = \{\{4,1,2\}, \{4,1,3\}\}$. The hypergraph $H$ is isomorphic to $\tprod{\mb{s}}{\mb{r}}$ as uncolored hypergraph, and thus usable in uncolored weak saturation. This shows that for $n_1 = n_2 \ge 4$,
\[
\wsat(\tprod{\mb{s}}{\mb{n}}, \tprod{\mb{s}}{\mb{r}}) < \wdir(\tprod{\mb{s}}{\mb{n}}, \tprod{\mb{s}}{\mb{r}}).
\]
\end{example}

However, the assumption \ref{item4:wsat-to-cwsat} may be dropped in the following special case.

\begin{theorem}\label{thrm:wsat-eq-wdir-constant-single-s}
Let $d \ge 1$ and $s\in \Nz$ be integers, let $\mb{n} \in \Nz^d$ be an integer vector, and let $\mc{R} \subseteq \Nz^d$ be a non-empty finite family of integer vectors. Suppose that for all $i \in [d]$ and $\mb{r}\in \mc{R}$ we have $n_i \ge s$ and $r_i \ge s$. Define $\mb{s} \in \Nz^d$ by $s_i = s$ for all $i \in [d]$. Then the following hold:
\begin{itemize}
  \item If $s \neq 0$, then $\Conv(\{\mb{s}\}) = S_d$ and
  \[
  \wsat(\tprod{\mb{s}}{\mb{n}}, \tprod{\mb{s}}{\mc{R}}) = \wdir(\tprod{\mb{s}}{\mb{n}}, \tprod{\mb{s}}{\fconv{\mc{R}}{S_d}}).
  \]
  \item If $s = 0$, then
  \[
  \wsat(\tprod{\mb{s}}{\mb{n}}, \tprod{\mb{s}}{\mc{R}}) =
  \begin{cases}
    0, & \text{if } \exists \mb{r} \in \mc{R},\ \sum_{i\in[d]} n_i \ge \sum_{i\in[d]} r_i;\\
    1, & \text{otherwise}.
  \end{cases}
  \]
\end{itemize}
\end{theorem}

\begin{proof}
The case $s = 0$ is immediate from the fact that the condition $\exists \mb{r}\in \mc{R}\ \sum_{i\in[d]} n_i \ge \sum_{i\in[d]}r_i$ is equivalent to the existence of a copy of some hypergraph from $\tprod{\mb{s}}{\mc{R}}$ inside $\tprod{\mb{s}}{\mb{n}}$.

Now assume $s \ne 0$. Let $\mb{r} \in \mc{R}$, and let $H$ be a uncolored copy of $\tprod{\mb{s}}{\mb{r}}$ in $\tprod{\mb{s}}{\mb{n}}$. We will show that $H$ is a colored copy of $\tprod{\mb{s}}{\conv{\mb{r}}{f}}$ for some permutation $f : [d] \to [d]$.

Write $V(H) = R = \bigdisjointcup_{i\in[d]} R_i$, where $\fbr{R_i}_{i\in[d]}$ are the preimages of parts in $\tprod{\mb{s}}{\mb{r}}$. For every edge $W \in E(H)$ we have
\begin{equation}\label{eq:S-only-const-1}
\forall i\in[d]\quad |N_i \cap W| = s.
\end{equation}

Let $A = \{ i \in [d] \mid r_i \ge s + 1\}$. Suppose for some $i \in A$ and $j \in [d]$ we have $R_i \cap N_j \neq \emptyset$. Then $R_i \subseteq N_j$, otherwise pick $x,y \in R_i$ with $x \in N_j$ and $y \notin N_j$. Choose an edge $W \in E(H)$ with $x \in W_i$, $y \notin W_i$. Then $W' = (W \setminus \{x\}) \cup \{y\}$ would also be an edge of $H$, but then $|W \cap N_j| \neq |W' \cap N_j|$, contradicting \eqref{eq:S-only-const-1}.

Therefore we have an injective function $f : A \to [d]$ with $ R_i \subseteq N_{f(i)}$ for all $i\in A$. Because \eqref{eq:S-only-const-1} holds, we also have
\[
\bigcup_{i \in [d] \setminus A} R_i \subseteq \bigcup_{i \in [d] \setminus f(A)} (R \cap N_i),
\]
and since $|R| = \sum_{i\in[d]} |R \cap N_i|$ and $f$ is injective, one deduces that
\[\forall i \in [d] \setminus f(A)\quad |R \cap N_i| = s.\]

We extend $f$ arbitrarily to a permutation of $[d]$. Then, for each $i\in[d]$, we have \mbox{$|R \cap N_{f(i)}| = r_i$}, and by \eqref{eq:S-only-const-1}, $H$ is a colored copy of $\tprod{\mb{s}}{\conv{\mb{r}}{f}}$
\end{proof}

Combining Theorems \ref{thrm:wsat-eq-wdir-constant-single-s} and \ref{thrm:main-theorem}, we obtain the following corollary, which establishes Theorem \ref{thrm:sym-Kn1n2-Kr1r2}.

\begin{corollary}\label{corollary:wsat-constant-single-s}
    Let $d \ge 1$ and $s\ge 1$ be integers, let $\mb{n} \in \Nz^d$ be an integer vector, and let $\mc{R} \subseteq \Nz^d$ be a non-empty finite family of integer vectors. Suppose that for all $i \in [d]$ and $\mb{r}\in \mc{R}$ we have $n_i \ge s$ and $r_i \ge s$. Define $\mb{s} \in \Nz^d$ by $s_i = s$ for all $i \in [d]$. Then
    \[
    \wsat(\tprod{\mb{s}}{\mb{n}}, \tprod{\mb{s}}{\mc{R}}) = \prod_{i\in[d]} \binom{n_i}{s} - \tilde{q}(\mb{n}, s, \mc{R}),
    \]
    where
    \[
        \tilde{q}(\mb{n}, s, \mc{R}) = \lbr{\fbr{T \in \timesprod_{i \in [d]} \binom{[n_i]}{s}\ \middle|\ \exists \mb{r}\in\mc{R}\ \exists \sigma \in S_d\ \forall i \in [d]\ T_i \in \binom{[n_i] \setminus [r_{\sigma(i)} - s]}{s}}}.
    \]
\end{corollary}
\begin{proof}
    By Theorems \ref{thrm:wsat-eq-wdir-constant-single-s} and \ref{thrm:main-theorem}, we have
    \[
    \wsat(\tprod{\mb{s}}{\mb{n}}, \tprod{\mb{s}}{\mc{R}})
     = \prod_{i\in[d]} \binom{n_i}{s} - q(\mb{n},\mb{s},\fconv{\mc{R}}{S_d}).
    \]
    Moreover, the identity
    \[
    q(\mb{n},\mb{s},\fconv{\mc{R}}{S_d})= \lbr{\fbr{T \in \timesprod_{i \in [d]} \binom{[n_i]}{s}\ \middle|\ \exists \mb{r}\in\mc{R}\ \exists \sigma \in S_d\ \forall i \in [d]\ T_i \in \binom{[n_i] \setminus [r_{\sigma(i)} - s]}{s}}}
    \]
    follows from the inclusion--exclusion principle, by the same argument as in Subsection~\ref{subsec:upper-bound-largest-s}.
\end{proof}

\section{Conclusion}\label{sec:conclusions}

A natural question arising from Theorem~\ref{thrm:main-theorem} is whether the lower bound in equation~\eqref{eq:main-theorem} is always tight. As the following example shows, this is not always the case.

\begin{example}
Let $d = 2$, $\mc{S} = \{(1,0),\ (0,1)\}$, and $\mc{R} = \{(2,1),\ (1,2)\}$. It is easy to show that for every $\mb{n} \in \Nz^2$ with $n_1 \ge 2$ and $n_2 \ge 2$,
\[
\wdir(\tprod{\mc{S}}{\mb{n}},\ \tprod{\mc{S}}{\mc{R}}) = 2.
\]
But
\[
q(\mb{n}, \mc{S}, \mc{R}) = (n_1 - 1) + (n_2 - 1) + 1,
\]
so the lower bound from Theorem~\ref{thrm:main-theorem} gives only
\[
\wdir(\tprod{\mc{S}}{\mb{n}},\ \tprod{\mc{S}}{\mc{R}}) \ge 1.
\]
\end{example}

This leads to the following question.

\begin{question}
Let $d$, $\mb{n}$, $\mc{S}$, and $\mc{R}$ be as in Theorem \ref{thrm:main-theorem}. What are the necessary and sufficient conditions for the equality
\[
\wdir(\tprod{\mc{S}}{\mb{n}},\ \tprod{\mc{S}}{\mc{R}}) = \sum_{\mb{s} \in \mc{S}} \prod_{i \in [d]} \binom{n_i}{s_i} - q(\mb{n}, \mc{S}, \mc{R})
\]
to hold?
\end{question}

More generally, the following problem arises.

\begin{question}
Let $d$, $\mb{n}$, $\mc{S}$, and $\mc{R}$ be as in Theorem \ref{thrm:main-theorem}. What is the value of
\[\wdir(\tprod{\mc{S}}{\mb{n}},\ \tprod{\mc{S}}{\mc{R}})?\]
\end{question}

Another interesting direction is to study the relationship between colored and uncolored weak saturation numbers, as partially discussed in Section~\ref{sec:wsat-eq-wdir}.

\begin{question}
Let $d$, $\mb{n}$, $\mc{S}$, and $\mc{R}$ be as in Theorem \ref{thrm:main-theorem}. Under what conditions do we have
\[
\wsat(\tprod{\mc{S}}{\mb{n}},\ \tprod{\mc{S}}{\mc{R}}) = \wdir\bigl(\tprod{\mc{S}}{\mb{n}},\ \tprod{\mc{S}}{\fconv{\mc{R}}{\Conv(\mc{S}) \cap S_d}}\bigr)?
\]
\end{question}





    \printbibliography

\end{document}